\documentclass[12pt]{article}
\usepackage{latexsym,amsfonts,amsmath}
\usepackage{amssymb}
\usepackage[usenames]{color}

\newtheorem{theorem}{Theorem} 
\newtheorem{lemma}{Lemma}

\newtheorem{defi}{Definition}

\newcommand{\satop}[2]{\stackrel{\scriptstyle{#1}}{\scriptstyle{#2}}}

\newcommand{\NN}{{\mathbb N}}

\newcommand{\ZZ}{{\mathbb Z}}
\newcommand{\RR}{{\mathbb R}}

\newcommand{\bsx}{\boldsymbol{x}}
\newcommand{\bsy}{\boldsymbol{y}}

\newcommand{\bsz}{\boldsymbol{z}}
\newcommand{\bsh}{\boldsymbol{h}}
\newcommand{\bsk}{\boldsymbol{k}}

\newcommand{\bsgamma}{\boldsymbol{\gamma}}

\newcommand{\cP}{\mathcal{P}}
\newcommand{\cQ}{\mathcal{Q}}
\newcommand{\cR}{\mathcal{R}}

\newcommand{\npoint}{\bsx_0,\ldots,\bsx_{N-1}}
\newcommand{\icomp}{\mathtt{i}}
\newcommand{\rd}{\,{\rm d}}
\newcommand{\uu}{\mathfrak{u}}
\newcommand{\bszero}{{\boldsymbol{0}}}
\newcommand{\gdisc}{D_{N,\bsgamma}^{\ast}}
\newcommand{\disc}{D_N^{\ast}}
\newenvironment{proof}{\begin{trivlist}
    \item[\hskip\labelsep{\it Proof.}]}{$\hfill\Box$\end{trivlist}}

\definecolor{darkred}{RGB}{139,0,0}
\definecolor{darkgreen}{RGB}{0,100,0}
\definecolor{darkmagenta}{RGB}{139,0,139}

\begin{document}
\title{The weighted star discrepancy of Korobov's $p$-sets}
\author{Josef Dick\thanks{J. D. is supported by a QEII Fellowship of the Australian Research Council.} and Friedrich Pillichshammer\thanks{F.P. is supported by the Austrian Science Fund (FWF): Project F5509-N26, which is a part of the Special Research Program ``Quasi-Monte Carlo Methods: Theory and Applications''.}}
\date{}
\maketitle

\begin{abstract}
We analyze the weighted star discrepancy of so-called $p$-sets which go back to definitions due to Korobov in the 1950s and Hua and Wang in the 1970s. Since then, these sets have largely been ignored since a number of other constructions have been discovered which achieve a better convergence rate. However, it has recently been discovered that the $p$-sets perform well in terms of the dependence on the dimension. 

We prove bounds on the weighted star discrepancy of the $p$-sets which hold for any choice of weights. For product weights we give conditions under which the discrepancy bounds are independent of the dimension $s$. This implies strong polynomial tractability for the weighted star discrepancy. We also show that a very weak condition on the product weights suffices to achieve polynomial tractability.
\end{abstract}

\centerline{\begin{minipage}[hc]{150mm}{
}
\end{minipage}}

\section{Introduction}\label{sec1}

For an $N$-element point set $\cP=\{\npoint\}$ in the $s$-dimensional unit cube $[0,1)^s$ the {\it discrepancy function} $\Delta$ is defined by
$$\Delta(\alpha_1,\ldots
,\alpha_s):=\frac{A_N\left(\prod_{i=1}^s[0,\alpha_i)\right)}{N}-
\alpha_1\cdots \alpha_s$$ for $0 < \alpha_1,\ldots, \alpha_s \le 1$.
Here $A_N(E)$ denotes the number of indices $n \in \{0,1,\ldots,N-1\}$,
such that $\bsx_n$ belongs to the set $E$. By taking the sup
norm of this function, we obtain the {\it star discrepancy}
\begin{eqnarray*}
\disc(\cP) =\sup_{\bsz \in [0,1]^s} |\Delta(\bsz)|
\end{eqnarray*}
of the point set $\cP$. 

The motivation for the definition of the star discrepancy comes from quasi-Monte Carlo integration $\frac{1}{N} \sum_{n=0}^{N-1} f(\bsx_n) \approx \int_{[0,1]^s} f(\bsx) \,\mathrm{d} \bsx$ of functions over the $s$-dimensional unit cube $[0,1]^s$. There one requires point sets $\cP$ which are very well distributed. In many cases the quality of the distribution of a point set is measured by the star discrepancy $\disc(\cP)$, which is intimately linked to the worst-case error of quasi-Monte Carlo integration via the well known Koksma-Hlawka inequality $$\left| \int_{[0,1]^s} f(\bsx) \,\mathrm{d} \bsx - \frac{1}{N} \sum_{n=0}^{N-1} f(\bsx_n) \right| \le \disc(\cP) V(f),$$ where $V(f)$ is the variation of $f$ in the sense of Hardy and Krause. See, for example, \cite{DP10,dt,kuinie,leopil,niesiam} for more information. 

At the end of the 1990s Sloan and Wo\'zniakowski~\cite{SW98} (see also \cite{DSWW,NW10})
introduced the notion of weighted discrepancy and proved a ``weighted'' Koksma-Hlawka inequality. The idea is that in many applications some projections are more important than others and that this should also be reflected in the quality measure of the point set. 

We start with some notation which goes back to the paper \cite{SW98}: let $[s]=\{1,2,\ldots ,s\}$ denote the set of coordinate indices.
For $\uu \subseteq [s]$, $\uu \not= \emptyset$, let $\gamma_{\uu}$
be a nonnegative real number (the {\it weight corresponding to the group of variables given by $\uu$}), $|\uu|$ the cardinality
of $\uu$, and for a vector $\bsz \in [0,1]^s$ let $\bsz_{\uu}$ denote
the vector from $[0,1]^{|\uu|}$ containing the components of $\bsz$
whose indices are in $\uu$. By $(\bsz_{\uu},1)$ we mean the vector
$\bsz$ from $[0,1]^s$ with all components whose indices are not in
$\uu$ replaced by 1.

\begin{defi}\label{weighted_star_dis}\rm
For an $N$-element point set $\cP$ in $[0,1)^s$ and given weights
$\bsgamma=\{\gamma_{\uu} : \emptyset \not= \uu \subseteq [s]\}$,
the {\it weighted star discrepancy} $\gdisc$ is given by
\begin{eqnarray*}
\gdisc(\cP) =\sup_{\bsz \in [0,1]^s} \max_{\emptyset \not=\uu \subseteq [s]}
\gamma_{\uu} |\Delta(\bsz_{\uu},1)|.
\end{eqnarray*}
\end{defi}
If $\gamma_{\uu} = 1$ for all $\uu \subseteq [s]$ (or $\gamma_{[s]}=1$ and $\gamma_{\uu}=0$ for $\uu \subsetneq [s]$), then the weighted star discrepancy coincides with the classical star discrepancy. 

The most popular and studied weights in literature are so-called {\it product weights} which are weights of the form $\gamma_{\uu}=\prod_{j \in \uu}\gamma_{j}$, for $\emptyset \not=\uu \subseteq [s]$, where the $\gamma_{j}$'s are positive reals, the weights associated with the $j$th component. See, for example, \cite{DSWW, SW98}. We assume throughout the paper that the weights $\gamma_j$ are non-increasing, i.e., $\gamma_1 \ge \gamma_2 \ge \gamma_3 \ge \ldots$.

\subsection*{Tractability}

The dependence on the dimension is the subject of tractability studies \cite{NW08, NW10, NW12}. We introduce the necessary background in the following.

For $s, N \in \NN$ the {\it $N$th minimal weighted star discrepancy} is  $${\rm disc}_{\bsgamma}(N,s)=\inf_{\cP \subseteq [0,1)^s \atop \# \cP=N} D_{N,\bsgamma}^{\ast}(\cP).$$

We would like to have a point set in the $s$-dimensional unit cube with weighted star discrepancy of at most $\varepsilon \in (0,1)$ and we are looking for the smallest cardinality $N$ of a point set such that this can be achieved. For $\varepsilon \in (0,1)$ and dimension $s \in \NN$ we define the {\it information complexity} $$N_{\min}(\varepsilon,s):=\min\{N \in \NN\, : \, {\rm disc}_{\bsgamma}(N,s) \le \varepsilon\},$$  which is sometimes also called the {\it inverse of the weighted star discrepancy}.

\begin{defi}\rm
\begin{enumerate}
\item We say that the weighted star discrepancy is {\it polynomially tractable}, if there exist nonnegative real numbers $C,\alpha$ and $\beta$ such that 
\begin{equation}\label{pt}
N_{\min}(\varepsilon,s) \le C s^{\alpha} \varepsilon^{-\beta}
\end{equation}
holds for all dimensions $s\in \NN$ and for all $\varepsilon \in (0,1)$. The infima over all $\alpha,\beta>0$ such that \eqref{pt} holds are called the $s$-exponent and the $\varepsilon$-exponent, respectively, of polynomial tractability. 
\item We say that the weighted star discrepancy is {\it strongly polynomially tractable}, if there exist nonnegative real numbers $C$ and $\alpha$ such that 
\begin{equation}\label{spt}
N_{\min}(\varepsilon,s) \le C \varepsilon^{-\alpha}
\end{equation}
holds for all dimensions $s\in \NN$ and for all $\varepsilon \in (0,1)$. The infimum over all $\alpha>0$ such that \eqref{spt} holds is called the $\varepsilon$-exponent of strong polynomial tractability. 
\end{enumerate} 
\end{defi}

Polynomial tractability means that there exists a point set whose cardinality is polynomial in $s$ and $\varepsilon^{-1}$ such that the weighted star discrepancy of this point set is bounded by $\varepsilon$. Polynomial tractability and strong polynomial tractability for the classical star discrepancy are defined in the same manner as in the weighted case.

An excellent survey on tractability of different notions of discrepancy can be found in the paper \cite{novwoz} or in the books \cite{NW08,NW10,NW12} which perfectly summarize the current state of the art in tractability theory.

\subsection*{Results on the weighted star-discrepancy}

We provide an informal description of our results. The details are given in Section~\ref{sec2}. We study three kinds of point sets in the unit cube which go back to Korobov, and Hua and Wang, and which are sometimes called the ``$p$-sets''. The advantage of these point sets is that their constructions are very easy (see Section~\ref{sec_pset} for details).  As an example, one of these $p$-sets is given by the points $(\{n/p\},\{n^2/p\},\ldots,\{n^s/p\})$ for $n=0,1,\ldots,p-1$, where $p$ is a prime number, and where $\{x\} = x - \lfloor x \rfloor$ denotes the fractional part of $x$ for positive real numbers $x$.

For simplicity we restrict ourselves to product weights. Results for general weights are given in Section~\ref{sec2}. In this paper we show that if the product weights $\bsgamma = \{\prod_{j \in \uu} \gamma_j\}$ satisfy
\begin{equation*}
\sum_{j=1}^\infty \gamma_j < \infty,
\end{equation*}
then for any $0 < \delta < 1/2$ there exists a constant $c^{(1)}_{\bsgamma, \delta} > 0$ which depends only on $\bsgamma$ and $\delta$ but not on the number of points $N$ and the dimension $s$, such that Korobov's $p$-sets $\cP$ satisfy
\begin{equation*}
\gdisc(\cP) \le c^{(1)}_{\bsgamma, \delta} \frac{1}{N^{1/2-\delta}}.
\end{equation*}
This implies strong polynomial tractability. If there exists a real number $t > 0$ such that the product weights $\bsgamma = \{\prod_{j \in \uu} \gamma_j\}$ satisfy
\begin{equation*}
\sum_{j=1}^\infty \gamma_j^t < \infty,
\end{equation*}
then for any $0 < \delta < 1/2$ there exists a constant $c^{(2)}_{\bsgamma, \delta, t} > 0$ which depends only on $\bsgamma$,  $\delta$ and $t$, but not on the number of points $N$ and the dimension $s$, such that Korobov's $p$-sets $\cP$ satisfy
\begin{equation*}
\gdisc(\cP) \le  c^{(2)}_{\bsgamma, \delta, t} \frac{s}{N^{1/2-\delta}}.
\end{equation*}
In this case we have polynomial tractability.

\subsection*{Literature review}

To put our results into context, we provide a review of known results. From Heinrich, Novak, Wasilkowski, and Wo\'{z}niakowski~\cite{hnww} it is known that for any number of points $N$ and dimension $s$ there exists a point set $\cP_{N,s}\subseteq [0,1)^s$, such that 
\begin{equation}\label{HNWW}
D_N^{\ast}(\cP_{N,s}) \le C \sqrt{\frac{s}{N}} 
\end{equation}
for some constant $C>0$.  Hence the classical star discrepancy is tractable with $s$-exponent at most one and $\varepsilon$-exponent at most two. It was further shown in \cite{hnww} that the inverse of the classical star discrepancy is at least $c s \log \varepsilon^{-1}$ with an absolute constant $c>0$ for all $\varepsilon \in (0,\varepsilon_0]$ and $s \in \NN$. This lower bound was improved by Hinrichs \cite{hin} to $c s \varepsilon^{-1}$ with an absolute constant $c>0$ for all $\varepsilon \in (0,\varepsilon_0]$ and $s \in \NN$. From these results it follows, that the classical star 
discrepancy cannot be strongly polynomially tractable. We stress that all mentioned results are non-constructive. A first constructive approach is given in \cite{dgs}. However here for given $s$ and $\varepsilon$ the authors can only ensure a running time for the construction algorithm of order $C^s s^s (\log s)^s \varepsilon^{-2(s+2)}$ which is too expensive for practical applications. 

The bound \eqref{HNWW} can be interpreted as having product weights $(\gamma_j)$ for which $\gamma_j = 1$ for all $j$. In comparison, to achieve polynomial tractability in our result, we require that for product weights we have $\sum_{j = 1}^\infty \gamma_j^t < \infty$ for some arbitrarily large real number $t > 0$. Note that \cite{hnww} only show an existence result, whereas our result is completely constructive.

We recall some known results for the weighted star discrepancy. The first result was shown in \cite{HPS2008} (see also \cite{DP10,NW10} for a summary).
\begin{theorem}[Hinrichs, Pillichshammer, Schmid]\label{thHPS}
There exists a constant $C>0$ with the following property: for given number of points $N$ and dimension $s$ there exists an $N$-element point set $\cP$ in $[0,1)^s$ such that
\begin{equation}\label{bdx1}
\gdisc(\cP) \le C\,\frac{1+\sqrt{\log s}}{\sqrt{N}}  \max_{\emptyset \not=\uu
\subseteq [s]} \gamma_{\uu} \sqrt{|\uu|}.
\end{equation}
\end{theorem}
Note that the point set $\cP$ from Theorem~\ref{thHPS} is independent of the choice of weights. The result is a pure existence result. Under very mild conditions on the weights Theorem~\ref{thHPS} implies polynomial tractability with $s$-exponent zero. See \cite{HPS2008} for details.

For product weights we have the following result which is taken from \cite{DLP2006} (see also \cite[Corollary~8]{dnp06} and \cite[Corollary~10.30]{DP10}):

\begin{theorem}[Dick, Leobacher, Pillichshammer]\label{thHPS2}
For every prime number $p$, every $m \in \NN$ and for given product weights $(\gamma_j)_{j \ge 1}$ with $\sum_j \gamma_j < \infty$ one can construct (component-by-component) a $p^m$-element point set $\cP$ in $[0,1)^s$ such that for every $\delta >0$ there exists a quantity $C_{\bsgamma,\delta}>0$ with the property $$D_{p^m,\bsgamma}^{\ast}(\cP)\le \frac{C_{\bsgamma,\delta}}{p^{m(1-\delta)}}.$$
\end{theorem}

Note that the point set $\cP$ from Theorem~\ref{thHPS2} depends on the choice of weights. The result implies that the weighted star discrepancy is strongly polynomially tractable with $\varepsilon$-exponent equal to one, as long as the weights $\gamma_j$ are summable. See \cite{DLP2006,dnp06,DP10,HPS2008} for more details.\\

The next result is about Niederreiter sequences in prime-power base $q$. For the definition of Niederreiter sequences we refer to \cite{DP10,niesiam}. The following result is \cite[Lemma~1]{wang2}:

\begin{lemma}[Wang]
For $N \in \NN$ let $\cP$ be the first $N$-elements of a Niederreiter sequence in prime-power base $q$. For $\uu \subseteq [s]$ we denote by $\cP_{\uu}$ the $|\uu|$-dimensional point set consisting of the projections of the elements of $\cP$ onto the coordinates which belong to $\uu$. Then we have $$D_N^{\ast}(\cP_{\uu}) \le \frac{1}{N} \prod_{j \in \uu}(C j \log(j+q) \log(qN)),$$ where $C>0$ is an absolute constant which is independent of $\uu$ and $s$.
\end{lemma}

Similar results can be shown for Sobol' sequences and for the Halton sequence (see \cite{wang1,wang2}). From this result one obtains:

\begin{theorem}\label{thm_afterWang}
For the weighted star discrepancy of the first $N$ elements $\cP$ of an $s$-dimensional Niederreiter sequence in prime-power base $q$ we have $$D_{N,\bsgamma}^{\ast}(\cP) \le \frac{1}{N} \max_{\emptyset \not= \uu \subseteq [s]} \gamma_{\uu} \prod_{j \in \uu}  (C j \log(j+q) \log(qN)).$$
\end{theorem}

In the case of product weights one can easily deduce from Theorem~\ref{thm_afterWang} that the weighted star discrepancy of the Niederreiter sequence can be bounded independently of the dimension whenever the weights satisfy $\sum_j \gamma_j j \log j < \infty$. This implies strong polynomial tractability with $\varepsilon$-exponent equal to one. (The same result can be shown for Sobol' sequences and for the Halton sequence).\\

A comparison of the results presented in this section with the new results will be given at the end of Section~\ref{sec2}.

\section{Korobov's $p$-sets}\label{sec_pset}

Let $p$ be a prime number. We consider the following point sets in $[0,1)^s$: 
\begin{itemize}
\item Let $\cP_{p,s}=\{\bsx_0,\ldots,\bsx_{p-1}\}$ with $$\bsx_n=\left(\left\{\frac{n}{p}\right\},\left\{\frac{n^2}{p}\right\},\ldots,\left\{\frac{n^s}{p}\right\}\right)\ \ \ \mbox{ for }\ n=0,1,\ldots,p-1.$$ The point set $\cP_{p,s}$ was introduced by Korobov \cite{kor1963} (see also \cite[Section 4.3]{huawang}).
\item Let $\cQ_{p^2,s}=\{\bsx_0,\ldots,\bsx_{p^2-1}\}$ with $$\bsx_n=\left(\left\{\frac{n}{p^2}\right\},\left\{\frac{n^2}{p^2}\right\},\ldots,\left\{\frac{n^s}{p^2}\right\}\right)\ \ \ \mbox{ for }\ n=0,1,\ldots,p^2-1.$$ The point set $\cQ_{p,s}$ was introduced by Korobov \cite{kor1957} (see also \cite[Section 4.3]{huawang}).
\item Let $\cR_{p^2,s}=\{\bsx_{a,k}\ : \ a,k \in \{0,\ldots,p-1\}\}$ with $$\bsx_{a,k}=\left(\left\{\frac{k}{p}\right\},\left\{\frac{a k}{p}\right\},\ldots,\left\{\frac{a^{s-1} k}{p}\right\}\right)\ \ \ \mbox{ for }\ a,k=0,1,\ldots,p-1.$$ Note that $\cR_{p^2,s}$ is the multi-set union of all Korobov lattice point sets with modulus $p$. The point set $\cR_{p^2,s}$ was introduced by Hua and Wang (see \cite[Section 4.3]{huawang}).
\end{itemize}

Hua and Wang \cite{huawang} called the point sets  $\cP_{p,s}$, $\cQ_{p^2,s}$ and $\cR_{p^2,s}$ the {\it $p$-sets}.

\section{The weighted star discrepancy of the $p$-sets}\label{sec2}

The classical (i.e. unweighted) star discrepancy of the $p$-sets is studied in \cite[Theorem 4.7-4.9]{huawang}. Here we consider the weighted star discrepancy.

\begin{theorem}\label{thm1}
Let $p$ be a prime number. For arbitrary weights $\bsgamma=\{\gamma_{\uu} : \emptyset \not= \uu \subseteq [s]\}$ we have: 
\begin{align*}
D_{p,\bsgamma}^{\ast}(\cP_{p,s}) \le &  \frac{2}{\sqrt{p}} \max_{\emptyset \not= \uu \subseteq [s]} \gamma_{\uu} (\max \uu) \left(4 \log p\right)^{|\uu|},\\
D_{p^2,\bsgamma}^{\ast}(\cQ_{p^2,s}) \le & \frac{3}{p} \max_{\emptyset \not= \uu \subseteq [s]} \gamma_{\uu} (\max \uu) \left(6 \log p\right)^{|\uu|}, \ \ \mbox{ and }\\
D_{p^2,\bsgamma}^{\ast}(\cR_{p^2,s}) \le &\frac{2}{p} \max_{\emptyset \not= \uu \subseteq [s]} \gamma_{\uu} (\max \uu) \left(4 \log p\right)^{|\uu|}.
\end{align*}
\end{theorem}

The proof of Theorem~\ref{thm1} will be given in Section~\ref{prThm1}.

Note that the point sets $\cP_{p,s}, \cQ_{p^2,s}$ and $\cR_{p^2,s}$ from Theorem~\ref{thm1} are independent of the choice of weights.\\

Now we consider product weights and study tractability properties. Let $\gamma_{\uu}=\prod_{j \in \uu} \gamma_j$ where $\gamma_j >0$ for $j\in \NN$ and $\gamma_1 \ge \gamma_2 \ge \gamma_3 \ldots$. 

\begin{theorem}\label{thm2}
Assume that the weights $\gamma_j$ are non-increasing.
\begin{enumerate}
\item If $$\sum_{j=1}^{\infty} \gamma_j < \infty,$$ then for all $\delta>0$ there exist quantities $c_{\bsgamma,\delta}',c_{\bsgamma,\delta}'',c_{\bsgamma,\delta}'''>0$, which are independent of $p$ and $s$, such that
\begin{align*} 
D_{p,\bsgamma}^{\ast}(\cP_{p,s}) & \le \frac{c_{\bsgamma,\delta}'}{p^{1/2 -\delta}},\\
D_{p^2,\bsgamma}^{\ast}(\cQ_{p^2,s}) & \le \frac{c_{\bsgamma,\delta}''}{p^{1 -\delta}}, \ \ \mbox{ and }\\
D_{p^2,\bsgamma}^{\ast}(\cR_{p^2,s}) & \le \frac{c_{\bsgamma,\delta}'''}{p^{1 -\delta}}.
\end{align*}

\item If there exists a real number $t > 0$ such that $$\sum_{j=1}^{\infty} \gamma_j^t < \infty,$$ then for all $\delta>0$ there exist quantities $c_{\bsgamma,\delta, t}',c_{\bsgamma,\delta, t}'',c_{\bsgamma,\delta, t}'''>0$, which are independent of $p$ and $s$, such that
\begin{align*} 
D_{p,\bsgamma}^{\ast}(\cP_{p,s}) & \le \frac{c_{\bsgamma,\delta, t}' \ s}{p^{1/2 -\delta}},\\
D_{p^2,\bsgamma}^{\ast}(\cQ_{p^2,s}) & \le \frac{c_{\bsgamma,\delta, t}'' \ s}{p^{1 -\delta}}, \ \ \mbox{ and }\\
D_{p^2,\bsgamma}^{\ast}(\cR_{p^2,s}) & \le \frac{c_{\bsgamma,\delta, t}''' \ s}{p^{1 -\delta}}.
\end{align*}
\end{enumerate}
\end{theorem}

The proof of Theorem~\ref{thm2} will be given in Section~\ref{prThm2}.\\

Note that the results in Point 1. of Theorem~\ref{thm2} imply strong polynomial tractability for the weighted star discrepancy. We show this for the $p$-set $\cP_{p,s}$. Assume that $\sum_j  \gamma_j < \infty$. Fix $\delta>0$. For $\varepsilon>0$, let $p$ be the smallest prime number that is larger or equal to $\lceil (c_{\bsgamma,\delta} \varepsilon^{-1})^{\frac{2}{1-2\delta}} \rceil =:M$. Then we have $D_{p,\bsgamma}^{\ast}(\cP_{p,\bsgamma}) \le \varepsilon$ and hence $$N_{\min}(\varepsilon,s) \le p < 2 M =2 \lceil (c_{\bsgamma,\delta} \varepsilon^{-1})^{\frac{2}{1-2\delta}} \rceil ,$$ where we used Bertrand's postulate which tells us that $M \le p < 2 M$. Hence the weighted star discrepancy is strongly polynomially tractable with $\varepsilon$-exponent at most 2.\\

In the same way, the results in Point 2. of Theorem~\ref{thm2} imply polynomial tractability for the weighted star discrepancy. We show this for $\cP_{p,s}$. Assume that $\sum_j \gamma_j^t < \infty$ for some $t > 0$. Fix $\delta>0$. For $\varepsilon>0$, let $p$ be the smallest prime number that is larger or equal to $\lceil (s c_{\bsgamma,\delta, t} \varepsilon^{-1})^{\frac{2}{1-2\delta}} \rceil =:M$. Then we have $D_{p,\bsgamma}^{\ast}(\cP_{p,s}) \le \varepsilon$ and hence $$N_{\min}(\varepsilon,s) \le p < 2 M =2 \lceil (s c_{\bsgamma,\delta, t} \varepsilon^{-1})^{\frac{2}{1-2\delta}} \rceil ,$$ where we used again Bertrand's postulate which tells us that $M \le p < 2 M$. Hence the weighted star discrepancy is polynomially tractable with $s$-exponent and $\varepsilon$-exponent at most 2.\\

With the following table we put the result from Theorem~\ref{thm2} into the context of the known results from Section~\ref{sec1}. 
The second column ``point set $\cP$'' gives information about the point set, the column ``$\cP=\cP(\bsgamma)$'' indicates whether the point set depends on $\bsgamma$ or not, the column ``SPT'' displays the conditions on product weights under which strong polynomial tractability is achieved and the last column ``$\varepsilon$-exponent'' displays the respective $\varepsilon$-exponents of strong polynomial tractability.  
\begin{center}
\begin{tabular}{l||c|c|c|c}
 & point set $\cP$ & $\cP=\cP(\bsgamma)$ & SPT & $\varepsilon$-exponent  \\
\hline\hline
Theorem~\ref{thHPS} & existence & NO & not possible & --   \\
Theorem~\ref{thHPS2} & CBC & YES & $\sum_j \gamma_j < \infty$ & 1 \\
Theorem~\ref{thm_afterWang} & explicit & NO & $ \sum_j \gamma_j j \log j < \infty$ & 1 \\
Theorem~\ref{thm2} & explicit & NO &  $\sum_j \gamma_j < \infty $ & $\le 2$
\end{tabular}
\end{center}

\section{The proofs}

\subsection{Auxiliary results}

For $M \in \NN$, $M \ge 2$, put $C(M)=(-M/2,M/2]\cap \ZZ$ and $C_s(M)=C(M)^s$ the $s$-fold Cartesian product of $C(M)$. Further we write $C_s^{\ast}(M)=C_s(M)\setminus\{\bszero\}$. For $h \in C(M)$ put $r(h)=\max(1,|h|)$ and for $\bsh=(h_1,\ldots ,h_s)\in C_s(M)$ put $r(\bsh)=\prod_{j=1}^sr(h_j).$ The following result is from Niederreiter~\cite[Theorem~3.10]{niesiam} (in a slightly simplified form).

\begin{lemma}[Niederreiter]\label{le1}
For integers $M,N \ge 2$ and $\bsy_0,\ldots,\bsy_{N-1}\in \ZZ^s$, let $\cP=\{\npoint\}$ be the $N$-element point set consisting of the fractional parts $\bsx_n=\{\bsy_n/M\}$ for $n=0,\ldots,N-1$. Then we have $$D_N^{\ast}(\cP) \le \frac{s}{M} + \frac{1}{2}\sum_{\bsh \in C_s^{\ast}(M)} \frac{1}{r(\bsh)}\left|\frac{1}{N}\sum_{n=0}^{N-1}\exp(2 \pi \icomp \bsh \cdot \bsy_n /M)\right|,$$ where ``$\cdot$'' denotes the usual inner-product in $\RR^s$ and where $\icomp=\sqrt{-1}$.
\end{lemma} 

Now we extend this result to the weighted star discrepancy:

\begin{lemma}\label{le2}
For integers $M,N \ge 2$ and $\bsy_0,\ldots,\bsy_{N-1}\in \ZZ^s$, let $\cP=\{\npoint\}$ be the $N$-element point set consisting of the fractional parts $\bsx_n=\{\bsy_n/M\}$ for $n=0,\ldots,N-1$. Then we have
\begin{align*}
\gdisc(\cP) \le  \max_{\emptyset \not= \uu \subseteq [s]} \gamma_{\uu} \frac{|\uu|}{M}+ \max_{\emptyset \not= \uu \subseteq [s]} \gamma_{\uu} \sum_{\bsh \in C_{|\uu|}^{\ast}(M)} \frac{1}{r(\bsh)} \left|\frac{1}{N}\sum_{n=0}^{N-1} \exp(2 \pi \icomp \bsh \cdot \bsy_{n,\uu} /M) \right|,
\end{align*}
where $\bsy_{n,\uu} \in [0,1)^{|\uu|}$ is the projection of $\bsy_n$ to the coordinates given by $\uu$.
\end{lemma}

\begin{proof}
We have $$\gdisc(\cP)=\sup_{\bsz \in (0,1]^s} \max_{\emptyset \not=\uu \subseteq [s]} \gamma_{\uu} |\Delta_{\cP}((\bsz_{\uu},1))| \le \max_{\emptyset \not=\uu \subseteq [s]} \gamma_{\uu} \disc(\cP_{\uu}),$$ where $\cP_{\uu}=\{\bsx_{0,\uu},\ldots ,\bsx_{N-1,\uu}\}$ in $[0,1)^{|\uu|}$ consists of the points of $\cP$ projected to the components whose indices are in $\uu$. For any $\emptyset \not=\uu \subseteq [s]$ we have from Lemma~\ref{le1} that $$\disc(\cP_{\uu})\le \frac{|\uu|}{M} + \sum_{\bsh \in C_{|\uu|}^{\ast}(M)} \frac{1}{r(\bsh)} \left|\frac{1}{N}\sum_{n=0}^{N-1} \exp(2 \pi \icomp \bsh \cdot \bsy_{n,\uu} /M) \right|,$$ and the result follows.
\end{proof}

\subsection{The proof of Theorem~\ref{thm1}}\label{prThm1}

For the proof of Theorem~\ref{thm1} we use results which were already stated in the book of Hua and Wang \cite{huawang}. 

\begin{lemma}\label{le3}
Let $p$ be a prime number and let $s \in \NN$. Then for all $h_1,\ldots,h_s\in \ZZ$ such that $p \nmid h_j$ for at least one $j \in [s]$ we have $$\left|\sum_{n=0}^{p-1} \exp(2 \pi \icomp(h_1 n+h_2 n^2+\cdots+h_s n^s)/p)\right| \le (s-1) \sqrt{p}.$$ 
\end{lemma}

\begin{proof}
The result follows from a bound from A. Weil \cite{weil} on exponential sums which is widely known as {\it Weil bound}. For details we refer to \cite{dick2014}. 
\end{proof}

\begin{lemma}\label{le5}
Let $p$ be a prime number and let $s \in \NN$. Then for all $h_1,\ldots,h_s\in \ZZ$ such that $p \nmid h_j$ for at least one $j \in [s]$ we have $$\left|\sum_{n=0}^{p^2-1} \exp(2 \pi \icomp(h_1 n+h_2 n^2+\cdots+h_s n^s)/p^2)\right| \le (s-1) p.$$ 
\end{lemma}

\begin{proof}
See \cite[Lemma~4.6]{huawang}. 
\end{proof}

\begin{lemma}\label{le6}
Let $p$ be a prime number and let $s \in \NN$. Then for all $h_1,\ldots,h_s \in \ZZ$ such that $p \nmid h_j$ for at least one $j \in [s]$ we have $$\left|\sum_{a=0}^{p-1} \sum_{k=0}^{p-1} \exp(2 \pi \icomp k (h_1 +h_2 a+\cdots+h_s a^{s-1})/p)\right| \le (s-1) p.$$ 
\end{lemma}

\begin{proof}
Under the assumption that $p \nmid \gcd(h_1,\ldots,h_s)$, the number of solutions of the congruence $h_1 + h_2 x + \cdots + h_s x^{s-1} \equiv 0 \pmod{p}$ in $\{0,1,\ldots,p-1\}$ is at most $s-1$. Hence
\begin{align*}
\left|\sum_{a=0}^{p-1} \sum_{k=0}^{p-1} \exp(2 \pi \icomp k (h_1 +h_2 a+\cdots+h_s a^{s-1})/p)\right| & = \left|\sum_{a=0\atop h_1 +h_2 a+\cdots+h_s a^{s-1} \equiv 0 \pmod{p}}^{p-1} p\right|\\
& \le (s-1)p.
\end{align*}
\end{proof}

Now we can give the proof of Theorem~\ref{thm1}:

\begin{proof}
\begin{itemize}
\item We consider $\cP_{p,s}$: Here $\bsx_n$ is of the form $\bsx_n=\{\bsy_n/M\}$, where $\bsy_n=(n,n^2,\ldots,n^s) \in \ZZ^s$ and $M=p$. 

For $\emptyset \not= \uu \subseteq [s]$ we obtain from Lemma~\ref{le3} 
\begin{align*}
\sum_{\bsh \in C_{|\uu|}^{\ast}(p)} \frac{1}{r(\bsh)} \left|\frac{1}{p}\sum_{n=0}^{p-1} \exp(2 \pi \icomp \bsh \cdot \bsy_{n,\uu} /p) \right| & \le \frac{1}{p} \sum_{\bsh \in C_{|\uu|}^{\ast}(p)} \frac{1}{r(\bsh)} (\max \uu-1) \sqrt{p} \\
& \le \frac{\max \uu}{\sqrt{p}} \left(1+S_p\right)^{|\uu|},
\end{align*}
where $S_p=\sum_{h \in C_1^{\ast}(p)} |h|^{-1}$. A straight-forward estimate gives 
\begin{equation}\label{bdSN1}
S_p \le 2 \sum_{h=1}^{\lfloor p/2\rfloor} \frac{1}{h} \le 2 \left(1+\int_1^{p/2} \frac{\rd t}{t}\right) =2(1 + \log (p / 2)).
\end{equation}
Hence we obtain 
\begin{align*}
\sum_{\bsh \in C_{|\uu|}^{\ast}(p)} \frac{1}{r(\bsh)} \left|\frac{1}{p}\sum_{n=0}^{p-1} \exp(2 \pi \icomp \bsh \cdot \bsy_{n,\uu} /p) \right| & \le \frac{\max \uu}{\sqrt{p}} \left(2+2\log (p/2) \right)^{|\uu|},
\end{align*}
Inserting this into Lemma~\ref{le2} gives
\begin{align*}
D_{p,\bsgamma}^{\ast}(\cP_{p,s}) & \le  \max_{\emptyset \not= \uu \subseteq [s]} \gamma_{\uu}\frac{|\uu|}{p} + \max_{\emptyset \not= \uu \subseteq [s]} \gamma_{\uu} \frac{\max \uu}{\sqrt{p}} \left(2+2\log (p/2) \right)^{|\uu|}\\
& \le \frac{2}{\sqrt{p}} \max_{\emptyset \not= \uu \subseteq [s]} \gamma_{\uu} (\max \uu) \left(4 \log p\right)^{|\uu|}.
\end{align*}

\item We consider $\cQ_{p^2,s}$: Here $\bsx_n$ is of the form $\bsx_n=\{\bsy_n/M\}$, where $\bsy_n=(n,n^2,\ldots,n^s) \in \ZZ^s$. Let $M=p^2$.  

For $\bsh=(h_j)_{j \in \uu}\in \ZZ^{|\uu|}$ we write $p|\bsh$ if $p|h_j$ for all $j \in \uu$ and $p \nmid \bsh$ if this is not the case.

For $\emptyset \not= \uu \subseteq [s]$ we obtain from Lemma~\ref{le5} 
\begin{eqnarray*}
\lefteqn{\sum_{\bsh \in C_{|\uu|}^{\ast}(p^2)} \frac{1}{r(\bsh)} \left|\frac{1}{p^2}\sum_{n=0}^{p^2-1} \exp(2 \pi \icomp \bsh \cdot \bsy_{n,\uu} /p^2) \right|}\\
& \le & \sum_{\bsh \in C_{|\uu|}^{\ast}(p^2) \atop p|\bsh} \frac{1}{r(\bsh)}  \left|\frac{1}{p^2}\sum_{n=0}^{p^2-1} \exp(2 \pi \icomp \bsh \cdot \bsy_{n,\uu} /p^2) \right| +\frac{1}{p^2} \sum_{\bsh \in C_{|\uu|}^{\ast}(p^2)\atop p \nmid \bsh} \frac{1}{r(\bsh)}(\max \uu) p. 
\end{eqnarray*}
We consider now the first sum where $p \mid \bsh$. Let $\bsk = \bsh/p \in \mathbb{Z}^\uu$. Since $\bsy_{n,\uu} = (n^j)_{j \in \uu}$ we have from Lemma~\ref{le3} that
\begin{eqnarray*}
\left| \frac{1}{p^2}\sum_{n=0}^{p^2-1} \exp(2 \pi \icomp \bsh \cdot \bsy_{n,\uu} /p^2) \right| & = & \left| \frac{1}{p^2} \sum_{\ell=0}^{p-1} \sum_{n= \ell p}^{\ell p + p-1}  \exp(2 \pi \icomp \bsk \cdot \bsy_{n,\uu} /p) \right| \\ & = & \left|  \frac{1}{p} \sum_{n= 0}^{p-1}  \exp(2 \pi \icomp \bsk \cdot \bsy_{n,\uu} /p) \right| \\ & \le & \frac{\max \uu }{\sqrt{p}}.
\end{eqnarray*}
Further we have
\begin{eqnarray*}
\sum_{\satop{ \bsh \in C_{|\uu|}^{\ast}(p^2) }{p \mid \bsh} } \frac{1}{r(\bsh)} & \le & \frac{1}{p} \sum_{\bsh \in C_{|\uu|}^{\ast}(p^2)} \frac{1}{r(\bsh)} \le \frac{1}{p} (1+ S_{p^2})^{|\uu|},
\end{eqnarray*}
where $S_{p^2}=\sum_{h \in C_1^{\ast}(p^2)} |h|^{-1}$. As before we find that  
\begin{equation*}
S_{p^2} \le 2( 1 +  \log (p^2/2)).
\end{equation*}
Thus we have
\begin{equation*}
\sum_{\bsh \in C_{|\uu|}^{\ast}(p^2)} \frac{1}{r(\bsh)} \left|\frac{1}{p^2}\sum_{n=0}^{p^2-1} \exp(2 \pi \icomp \bsh \cdot \bsy_{n,\uu} /p^2) \right| \le \frac{2 \max \uu}{p} (2+ 2 \log (p^2/2) )^{|\uu|}
\end{equation*}
Inserting this into Lemma~\ref{le2} gives
\begin{align*}
D_{p^2,\bsgamma}^{\ast}(\cQ_{p^2,s}) & \le  \max_{\emptyset \not= \uu \subseteq [s]} \gamma_{\uu}\frac{|\uu|}{p^2} + \max_{\emptyset \not= \uu \subseteq [s]} \gamma_{\uu} \frac{2(\max \uu)}{p} \left(2+2\log (p^2/2) \right)^{|\uu|}\\
& \le \frac{3}{p} \max_{\emptyset \not= \uu \subseteq [s]} \gamma_{\uu} (\max \uu) \left(6 \log p\right)^{|\uu|}.
\end{align*}

\item We consider $\cR_{p^2,s}$: Here $\bsx_{a,k}$ is of the form $\bsx_{a,k}=\{\bsy_{a,k}/M\}$, where $\bsy_{a,k}=(a k,a^2 k,\ldots,a^{s-1} k) \in \ZZ^s$ and $M=p$. 

For $\emptyset \not= \uu \subseteq [s]$ we obtain from Lemma~\ref{le6} 
\begin{align*}
\sum_{\bsh \in C_{|\uu|}^{\ast}(p)} \frac{1}{r(\bsh)} \left|\frac{1}{p^2}\sum_{a=0}^{p-1} \sum_{k=0}^{p-1} \exp(2 \pi \icomp \bsh \cdot \bsy_{a,k,\uu} /p) \right| \le & \frac{1}{p^2} \sum_{\bsh \in C_{|\uu|}^{\ast}(p)} \frac{1}{r(\bsh)} (\max \uu - 1) p \\ 
 \le & \frac{\max \uu}{p}(2+2 \log (p/2) )^{|\uu|},
\end{align*}
where we used \eqref{bdSN1}. Inserting this into Lemma~\ref{le2} gives
\begin{align*}
D_{p^2,\bsgamma}^{\ast}(\cR_{p^2,s}) & \le  \max_{\emptyset \not= \uu \subseteq [s]} \gamma_{\uu}\frac{|\uu|}{p} + \max_{\emptyset \not= \uu \subseteq [s]} \gamma_{\uu} \frac{(\max \uu)}{p} \left(2+2\log (p/2) \right)^{|\uu|}\\
& \le \frac{2}{p} \max_{\emptyset \not= \uu \subseteq [s]} \gamma_{\uu} (\max \uu) \left(4 \log p\right)^{|\uu|}.
\end{align*}
\end{itemize}
\end{proof}

\subsection{The proof of Theorem~\ref{thm2}}\label{prThm2}

\begin{proof}
We give the proof only for $\cP_{p,s}$. The proofs for $\cQ_{p^2,s}$ and $\cR_{p^2,s}$ follow by the same arguments.

\begin{enumerate}
\item Let $\delta > 0$. Assume that $\gamma_1 \ge \gamma_2 \ge \ldots$ and $\sum_j \gamma_j < \infty$. For $k \in \mathbb{N}_0$ let $\Gamma_k := \sum_{j=k+1}^{\infty} \gamma_j < \infty$ and note that $\lim_k \Gamma_k=0$. 
Let  $k_0 \in \mathbb{N}$ be the smallest integer such that $$\Gamma_{k_0} < \frac{\delta}{8 \mathrm{e}},$$ where $\mathrm{e}=\exp(1)$. Note that $k_0$ depends on $\bsgamma$ and $\delta$, but not on $s$ and $p$. 

We have
\begin{equation*}
k \gamma_k \le \sum_{j=1}^k \gamma_j \le \Gamma_0,
\end{equation*}
which implies that for any $k \in \mathbb{N}$ we have
\begin{equation*}
\gamma_k \le \frac{\Gamma_0}{k} < \infty.
\end{equation*}
Thus for any finite set $\uu \subseteq \mathbb{N}$ we have $$\gamma_{\max \uu} \max \uu \le  \Gamma_0 < \infty.$$ Therefore
\begin{align*}
D_{p,\bsgamma}^{\ast}(\cP_{p,s})  \le & \frac{2}{\sqrt{p}}\max_{\emptyset \not= \uu \subseteq [s]} (\max \uu) \prod_{j\in \uu}  \left(4 \gamma_j\log p\right)\\  \le &  \frac{8 \Gamma_0 \log p}{\sqrt{p}} \max_{\emptyset \not= \uu \subseteq [s-1]} \prod_{j \in \uu}(4 \gamma_j \log p) \\ \le & \frac{8 \Gamma_0 \log p}{\sqrt{p}} \prod_{j=1}^\ell (4 \gamma_j \log p),
\end{align*}
where $\ell \in \mathbb{N}_0$ is such that $\gamma_\ell \ge \frac{1}{4\log p} > \gamma_{\ell+1}$. If $\frac{1}{4 \log p} > \gamma_{1}$ then we set $\ell = 0$ and set the empty product $\prod_{j=1}^0 (4 \gamma_j \log p)$ to $1$. For $\ell \le k_0$ we have
\begin{equation*}
D_{p,\bsgamma}^{\ast}(\cP_{p,s})  \le \frac{2 (4 \Gamma_0 \log p)^{k_0 + 1} }{\sqrt{p}}.
\end{equation*}
Now consider the case when $\ell > k_0$. For $k > k_0$ we have
\begin{equation*}
(k-k_0) \gamma_k \le \sum_{j=k_0+1}^k \gamma_j \le \Gamma_{k_0}
\end{equation*}
and therefore $\gamma_k \le \frac{\Gamma_{k_0}}{k-k_0}$. Thus we have
\begin{align*}
D_{p,\bsgamma}^{\ast}(\cP_{p,s})  \le & \frac{2  (4 \Gamma_0 \log p)^{k_0 + 1} }{\sqrt{p}} \prod_{j=k_0+1}^\ell (4\gamma_j \log p) \\ \le &  \frac{2 (4 \Gamma_0 \log p)^{k_0 + 1} }{\sqrt{p}}  \frac{(4 \Gamma_{k_0} \log p)^{\ell - k_0}}{ (\ell-k_0)! }.
\end{align*}
Using Stirling's formula we obtain
\begin{align*}
 \frac{(4\Gamma_{k_0} \log p)^{\ell - k_0} }{(\ell - k_0)!} \le & \frac{1}{\sqrt{2\pi (\ell - k_0)}} \left(\frac{4 \Gamma_{k_0} \mathrm{e} \log p}{\ell - k_0} \right)^{\ell - k_0} \\ \le &  \left(1 + \frac{4 \Gamma_{k_0} \mathrm{e} \log p}{\ell - k_0} \right)^{\ell - k_0}  \\ \le & \mathrm{e}^{4 \Gamma_{k_0} \mathrm{e} \log p} \\ = & p^{4\Gamma_{k_0} \mathrm{e}} \le p^{\delta/2}.
\end{align*}

There exists a constant $c_{\bsgamma, \delta} > 0$ depending on $\bsgamma$ and $\delta$ but not on $s$ and $p$, such that
\begin{equation*}
2 (4 \Gamma_0 \log p)^{k_0+1} \le c_{\bsgamma, \delta} p^{\delta/2} \quad \mbox{for all } p \in \mathbb{N}.
\end{equation*}
Thus we obtain
\begin{align*}
D_{p,\bsgamma}^{\ast}(\cP_{p,s})  \le & \frac{c_{\bsgamma, \delta}}{p^{1/2-\delta}}.
\end{align*}

\item We use the notation from above. Assume now that for some $t > 0$ we have
\begin{equation*}
\sum_{j=1}^\infty \gamma_j^t < \infty.
\end{equation*}
We set
\begin{equation*}
\Gamma_{h, t} = \left( \sum_{j=h+1}^\infty \gamma_j^t \right)^{1/t}
\end{equation*}
and we set $\Gamma_{h} = \Gamma_{h,1}$.

Let $h_0$ be the smallest integer such that
\begin{equation*}
\Gamma_{h_0, t} \le \frac{\delta}{8 \mathrm{e}^t t}.
\end{equation*}
We estimate $\max \uu$ by $s$ to obtain
\begin{align*}
D_{p,\bsgamma}^{\ast}(\cP_{p,s})  \le & \frac{2}{\sqrt{p}}\max_{\emptyset \not= \uu \subseteq [s]} (\max \uu) \prod_{j\in \uu}  \left(4 \gamma_j\log p\right)  \le   \frac{2s}{\sqrt{p}} \max_{\emptyset \not= \uu \subseteq [s]} \prod_{j \in \uu}(4 \gamma_j \log p) \\  \le &  \frac{2s}{\sqrt{p}} \prod_{j=1}^\ell (4 \gamma_j \log p) \le \frac{2s (4 \Gamma_0 \log p)^{h_0}}{\sqrt{p}} \prod_{j=h_0+1}^\ell (4 \gamma_j \log p),
\end{align*}
where $\ell$ is defined as above and where we set $\prod_{j=h_0+1}^\ell (4\gamma_j \log p) = 1$ if $\ell \le h_0$.

Now we have
\begin{equation*}
(h-h_0) \gamma_h^t \le \sum_{j=h_0+1}^h \gamma_j^t \le \Gamma_{h_0,t}^t
\end{equation*}
and therefore $$\gamma_h \le \frac{\Gamma_{h_0,t}}{(h-h_0)^{1/t}}.$$
\end{enumerate}

Assume that $\ell > h_0$. Then, using Stirling's formula again, we obtain
\begin{eqnarray*}
\prod_{j=h_0+1}^\ell (4\gamma_j \log p) & \le & \frac{(4 \Gamma_{h_0,t} \log p)^{\ell-h_0}}{((\ell-h_0)!)^{1/t}} \\ & \le & \frac{1}{(2\pi (\ell - h_0)^{1/(2t)}} \left(\frac{ 4 \Gamma_{h_0,t} \mathrm{e}^{1/t} \log p}{(\ell-h_0)^{1/t} } \right)^{\ell-h_0}   \\ & \le & \left( \left(1 + \frac{ 4 \Gamma_{h_0,t} \mathrm{e}^{1/t} \log p}{(\ell-h_0)^{1/t} } \right)^{(\ell-h_0)/t} \right)^t  \\ & \le & \left( \mathrm{e}^{4 \Gamma_{h_0,t} \mathrm{e}^t \log p} \right)^t \\ & \le & p^{4 t \Gamma_{h_0, t} \mathrm{e}^t} \\ & \le & p^{\delta/2}.
\end{eqnarray*}
The result now follows by the same arguments as in the previous case.
\end{proof}

\vspace{0.5cm}
\noindent{\bf Author's Addresses:}\\

\noindent Josef Dick, School of Mathematics and Statistics, The University of New South Wales, Sydney, NSW 2052, Australia.  Email: josef.dick@unsw.edu.au \\

\noindent Friedrich Pillichshammer, Institut f\"{u}r Analysis, Universit\"{a}t Linz, Altenbergerstra{\ss}e 69, A-4040 Linz, Austria. Email: friedrich.pillichshammer@jku.at

\end{document}